\documentclass[12pt, reqno, a4paper]{amsart}

\usepackage{amssymb}
\usepackage{hyperref}

\usepackage[english]{babel}

\theoremstyle{plain}

\newtheorem*{corollary}{Corollary}
\newtheorem{lemma}{Lemma}
\newtheorem{theorem}{Theorem}
\newtheorem*{conjecture}{Conjecture}

\theoremstyle{remark}
\newtheorem*{remark}{Remark}

\theoremstyle{definition}

\newtheorem{example}{Example}

\DeclareMathOperator{\Id}{Id}

\DeclareMathOperator{\id}{id}
\DeclareMathOperator{\ch}{char}

\DeclareMathOperator{\ad}{ad}

\DeclareMathOperator{\End}{End}
\DeclareMathOperator{\Aut}{Aut}

\DeclareMathOperator{\sign}{sign}
\DeclareMathOperator{\Alt}{Alt}
\DeclareMathOperator{\Ann}{Ann}

\DeclareMathOperator{\PIexp}{PIexp}

\begin{document}

\title[On a formula for the PI-exponent]{On a formula for the PI-exponent
of Lie algebras}

\author{A.\,S.~Gordienko}
\address{Vrije Universiteit Brussel, Belgium}
\email{alexey.gordienko@vub.ac.be} 

\begin{abstract}
We prove that one of the conditions in M.\,V.~Zaicev's formula for the PI-exponent
and in its natural generalization for the Hopf PI-exponent,
can be weakened. Using the modification of the formula, we prove that if a finite dimensional semisimple Lie algebra acts by derivations on 
a finite dimensional Lie algebra over a field of characteristic $0$,
then the differential PI-exponent coincides with the ordinary one.
Analogously, the exponent
of polynomial $G$-identities of a finite dimensional Lie algebra
with a rational action of a connected reductive affine algebraic group $G$
by automorphisms,
coincides with the ordinary PI-exponent. In addition, we 
provide a simple formula
for the Hopf PI-exponent and prove the existence of the Hopf PI-exponent itself
 for $H$-module Lie algebras whose solvable radical is nilpotent, assuming
 only the $H$-invariance of the radical, i.e.
 under weaker assumptions on the $H$-action, than in
the general case. As a consequence, we show that the analog of Amitsur's
conjecture holds for $G$-codimensions of all finite dimensional
Lie $G$-algebras whose solvable radical is nilpotent, for an arbitrary group $G$.
\end{abstract}


\thanks{
Supported by Fonds Wetenschappelijk Onderzoek~--- Vlaanderen Pegasus Marie Curie post doctoral fellowship (Belgium).
}

\maketitle

\section{Introduction}
The intensive study of polynomial identities and their numeric invariants
revealed the strong connection of the invariants with the structure of an algebra~\cite{Bahturin, DrenKurs, ZaiGia, ZaiLie}. If an algebra is endowed with a grading, an action of a Lie algebra
by derivations, an action of a group by automorphisms and anti-automorphisms, or an action of a Hopf algebra, it is natural to consider graded, differential, $G$- or $H$-identities~\cite{BahtZaiGradedExp,
BahtZaiSehgal, BereleHopf, Kharchenko}.

In 2002, M.\,V.~Zaicev~\cite{ZaiLie} proved a formula
for the PI-exponent of finite dimensional Lie algebras
over an algebraically closed field of characteristic $0$. It can be shown~\cite{ASGordienko2, ASGordienko5, GordienkoKochetov}
 that, under some assumpions, the natural generalization of the formula (see Subsection~\ref{SubsectionOrigFormulaHPIexp}) holds for the exponent
of graded, differential, $G$-, and $H$-identities too.
In Subsection~\ref{SubsectionModificationHPIexp} we prove that one of the conditions can be weakened,
which makes the formula easier to apply.

In~\cite{GordienkoKochetov}, the authors showed
that if a connected reductive affine algebraic group~$G$
 acts on a finite dimensional associative algebra~$A$
rationally by automorphisms,
then the exponent of $G$-identities coincides with the ordinary PI-exponent of $A$.
Also, if a finite dimensional semisimple Lie algebra acts on 
 a finite dimensional associative algebra by derivations,
 then the differential PI-exponent coincides with the ordinary one.
Using the modification of M.\,V.~Zaicev's formula, we prove the
analogous results for finite dimensional Lie algebras (Theorems~\ref{TheoremLieDerPIexpEqual}
and~\ref{TheoremLieGConnPIexpEqual} in Section~\ref{SectionApplLieDer}).

In Section~\ref{SectionHRTheSame} we consider finite dimensional $H$-module Lie algebras~$L$
such that the solvable radical of $L$ is nilpotent and $H$-invariant. We prove the analog of Amitsur's conjecture for such algebras $L$
and provide a simple formula for the Hopf PI-exponent of $L$.

\section{Polynomial $H$-identities and their codimensions}\label{SectionDerH}

Let $H$ be a Hopf algebra over a field $F$.
An algebra $A$
over $F$
is an \textit{$H$-module algebra}
or an \textit{algebra with an $H$-action},
if $A$ is endowed with a homomorphism $H \to \End_F(A)$ such that
$h(ab)=(h_{(1)}a)(h_{(2)}b)$
for all $h \in H$, $a,b \in A$. Here we use Sweedler's notation
$\Delta h = h_{(1)} \otimes h_{(2)}$ where $\Delta$ is the comultiplication
in $H$.
\begin{example}
If $M$ is an $H$-module, then $\End_F(M)$ is an associative $H$-module
algebra where $(h\psi)(v)=h_{(1)}\psi((Sh_{(2)})v)$ for all $h\in H$, $\psi \in \End_F(M)$, and $v\in M$.
(Here $S$ is the antipode of $H$.)
\end{example}

We refer the reader to~\cite{Abe, Danara, Montgomery, Sweedler}
   for an account
  of Hopf algebras and algebras with Hopf algebra actions.
  
     Let $F \lbrace X \rbrace$ be the absolutely free nonassociative algebra
   on the set $X := \lbrace x_1, x_2, x_3, \ldots \rbrace$.
  Then $F \lbrace X \rbrace = \bigoplus_{n=1}^\infty F \lbrace X \rbrace^{(n)}$
  where $F \lbrace X \rbrace^{(n)}$ is the linear span of all monomials of total degree $n$.
   Let $H$ be a Hopf algebra over a field $F$. Consider the algebra $$F \lbrace X | H\rbrace
   :=  \bigoplus_{n=1}^\infty H^{{}\otimes n} \otimes F \lbrace X \rbrace^{(n)}$$
   with the multiplication $(u_1 \otimes w_1)(u_2 \otimes w_2):=(u_1 \otimes u_2) \otimes w_1w_2$
   for all $u_1 \in  H^{{}\otimes j}$, $u_2 \in  H^{{}\otimes k}$,
   $w_1 \in F \lbrace X \rbrace^{(j)}$, $w_2 \in F \lbrace X \rbrace^{(k)}$.
We use the notation $$x^{h_1}_{i_1}
x^{h_2}_{i_2}\ldots x^{h_n}_{i_n} := (h_1 \otimes h_2 \otimes \ldots \otimes h_n) \otimes x_{i_1}
x_{i_2}\ldots x_{i_n}$$ (the arrangements of brackets on $x_{i_j}$ and on $x^{h_j}_{i_j}$
are the same). Here $h_1 \otimes h_2 \otimes \ldots \otimes h_n \in H^{{}\otimes n}$,
$x_{i_1} x_{i_2}\ldots x_{i_n} \in F \lbrace X \rbrace^{(n)}$. 

Note that if $(\gamma_\beta)_{\beta \in \Lambda}$ is a basis in $H$, 
then $F \lbrace X | H\rbrace$ is isomorphic to the absolutely free nonassociative algebra over $F$ with free formal  generators $x_i^{\gamma_\beta}$, $\beta \in \Lambda$, $i \in \mathbb N$.
 
    Define on $F \lbrace X | H\rbrace$ the structure of a left $H$-module
   by $$h\,(x^{h_1}_{i_1}
x^{h_2}_{i_2}\ldots x^{h_n}_{i_n})=x^{h_{(1)}h_1}_{i_1}
x^{h_{(2)}h_2}_{i_2}\ldots x^{h_{(n)}h_n}_{i_n},$$
where $h_{(1)}\otimes h_{(2)} \otimes \ldots \otimes h_{(n)}$
is the image of $h$ under the comultiplication $\Delta$
applied $(n-1)$ times, $h\in H$. Then $F \lbrace X | H\rbrace$ is \textit{the absolutely free $H$-module nonassociative algebra} on $X$, i.e. for each map $\psi \colon X \to A$ where $A$ is an $H$-module algebra,
there exists a unique homomorphism $\bar\psi \colon 
F \lbrace X | H\rbrace \to A$ of algebras and $H$-modules, such that $\bar\psi\bigl|_X=\psi$.
Here we identify $X$ with the set $\lbrace x^1_j \mid j \in \mathbb N\rbrace \subset F \lbrace X | H\rbrace$.

Consider the $H$-invariant ideal $I$ in $F\lbrace X | H \rbrace$
generated by the set \begin{equation}\label{EqSetOfHGen}
\bigl\lbrace u(vw)+v(wu)+w(uv) \mid u,v,w \in  F\lbrace X | H \rbrace\bigr\rbrace \cup\bigl\lbrace u^2 \mid u \in  F\lbrace X | H \rbrace\bigr\rbrace.
\end{equation}
 Then $L(X | H) := F\lbrace X | H \rbrace/I$
is \textit{the free $H$-module Lie algebra}
on $X$, i.e. for any $H$-module Lie algebra $L$ 
and a map $\psi \colon X \to L$, there exists a unique homomorphism $\bar\psi \colon L(X | H) \to L$
of algebras and $H$-modules such that $\bar\psi\bigl|_X =\psi$. 
 We refer to the elements of $L(X | H)$ as \textit{Lie $H$-polynomials}
 and use the commutator notation for the multiplication.


\begin{remark} If $H$ is cocommutative and $\ch F \ne 2$, then $L(X | H)$ is the ordinary
free Lie algebra with free generators $x_i^{\gamma_\beta}$, $\beta \in \Lambda$, $i \in \mathbb N$
where   $(\gamma_\beta)_{\beta \in \Lambda}$ is a basis in $H$, since the ordinary ideal of 
$F\lbrace X | H \rbrace$ generated by~(\ref{EqSetOfHGen})
is already $H$-invariant.
However, if $h_{(1)} \otimes h_{(2)} \ne h_{(2)} \otimes h_{(1)}$ for some $h \in H$,
we still have $$[x^{h_{(1)}}_i, x^{h_{(2)}}_j]=h[x_i, x_j]=-h[x_j, x_i]=-[x^{h_{(1)}}_j, x^{h_{(2)}}_i]
= [x^{h_{(2)}}_i, x^{h_{(1)}}_j]$$ in $L(X | H)$ for all $i,j \in\mathbb N$,
i.e. in the case $h_{(1)} \otimes h_{(2)} \ne h_{(2)} \otimes h_{(1)}$ the algebra $L(X | H)$ is not free as an ordinary Lie algebra.
\end{remark}

Let $L$ be an $H$-module Lie algebra for
some Hopf algebra $H$ over a field $F$.
 An $H$-polynomial
 $f \in L ( X | H )$
 is a \textit{$H$-identity} of $L$ if $\psi(f)=0$
for all homomorphisms $\psi \colon L(X|H) \to L$
of algebras and $H$-modules. In other words, $f(x_1, x_2, \ldots, x_n)$
 is a polynomial $H$-identity of $L$
if and only if $f(a_1, a_2, \ldots, a_n)=0$ for any $a_i \in L$.
 In this case we write $f \equiv 0$.
The set $\Id^H(L)$ of all polynomial $H$-identities
of $L$ is an $H$-invariant ideal of $L(X|H)$.

Denote by $V^H_n$ the space of all multilinear Lie $H$-polynomials
in $x_1, \ldots, x_n$, $n\in\mathbb N$, i.e.
$$V^{H}_n = \langle [x^{h_1}_{\sigma(1)},
x^{h_2}_{\sigma(2)}, \ldots, x^{h_n}_{\sigma(n)}]
\mid h_i \in H, \sigma\in S_n \rangle_F \subset L( X | H ).$$
(All long commutators in the article are left-normed, although this is not important
in this particular case in virtue of the Jacobi identity.)
The number $c^H_n(L):=\dim\left(\frac{V^H_n}{V^H_n \cap \Id^H(L)}\right)$
is called the $n$th \textit{codimension of polynomial $H$-identities}
or the $n$th \textit{$H$-codimension} of $L$.

The analog of Amitsur's conjecture for $H$-codimensions of $L$ can be formulated
as follows.

\begin{conjecture}  There exists
 $\PIexp^H(L):=\lim\limits_{n\to\infty}
 \sqrt[n]{c^H_n(L)} \in \mathbb Z_+$.
\end{conjecture}

We call $\PIexp^H(L)$ the \textit{Hopf PI-exponent} of $L$.

Here we list three important particular cases:

\begin{example} Every algebra $L$ is an $H$-module algebra
for $H=F$. In this case the $H$-action is trivial and we get ordinary polynomial identities and their codimensions. (See the original definition e.g. in~\cite{Bahturin}.) We write $c_n(L):= c_n^F(L)$, $\Id(L):= \Id^F(L)$,
$V_n(L):=V_n^F(L)$, $\PIexp(L)=\PIexp^F(L)$.
\end{example}

\begin{example}
If $H=FG$ where $FG$ is the group algebra of a group $G$, then 
an $H$-module algebra $L$ is an algebra with a $G$-action by automorphisms.
In this case we get \textit{polynomial $G$-identities} and
\textit{$G$-codimensions}.
We write $c^G_n(L):= c_n^{FG}(L)$, $\Id^G(L):= \Id^{FG}(L)$,
$V^G_n(L):=V_n^{FG}(L)$, $\PIexp^G(L)=\PIexp^{FG}(L)$.
Note that one can consider $G$-actions not only by automorphisms, but by anti-automorphisms too
and define polynomial $G$-identities and
$G$-codimensions in this case as well. (See e.g.~\cite[Section~1.2]{ASGordienko5}.)
\end{example}

\begin{example} If $H=U(\mathfrak g)$ where $U(\mathfrak g)$ is the universal enveloping
algebra of a Lie algebra $\mathfrak g$, then an $H$-module algebra is an algebra
with a $\mathfrak g$-action by derivations. The corresponding $H$-identities are called
\textit{differential identities} or \textit{polynomial identities with derivations}.
\end{example}

\section{Two formulas for the Hopf PI-exponent}

\subsection{$H$-nice Lie algebras}\label{SubsectionHnice}

The analog of Amitsur's conjecture was proved~\cite{ASGordienko5} for a wide class of $H$-module Lie algebras
that we call $H$-nice (see the definition below). The class of $H$-nice algebras includes finite dimensional semisimple $H$-module Lie algebras,
finite dimensional $H$-module Lie algebras for finite dimensional semisimple Hopf algebras $H$,
finite dimensional Lie algebras with a rational action of a reductive affine algebraic group
by automorphisms, and finite dimensional Lie algebras graded by an Abelian group
(see~\cite{ASGordienko5}).

Let $L$ be a finite dimensional $H$-module Lie algebra
where $H$ is a Hopf algebra over an algebraically closed field $F$
of characteristic $0$.
We say that $L$ is \textit{$H$-nice} if either $L$ is semisimple or the following conditions hold:
\begin{enumerate}
\item \label{ConditionRNinv}
the nilpotent radical $N$ and the solvable radical $R$ of $L$ are $H$-invariant;
\item \label{ConditionLevi} \textit{(Levi decomposition)}
there exists an $H$-invariant maximal semisimple subalgebra $B \subseteq L$ such that
$L=B\oplus R$ (direct sum of $H$-modules);
\item \label{ConditionWedderburn} \textit{(Wedderburn~--- Mal'cev decompositions)}
for any $H$-submodule $W \subseteq L$ and associative $H$-module subalgebra $A_1 \subseteq \End_F(W)$, 
the Jacobson radical $J(A_1)$ is $H$-invariant and
there exists an $H$-invariant maximal semisimple associative subalgebra $\tilde A_1 \subseteq A_1$
such that $A_1 = \tilde A_1 \oplus J(A_1)$ (direct sum of $H$-submodules);
\item \label{ConditionLComplHred}
for any $H$-invariant Lie subalgebra $L_0 \subseteq \mathfrak{gl}(L)$
such that $L_0$ is an $H$-module algebra and $L$ is a completely reducible $L_0$-module disregarding $H$-action, $L$ is a completely reducible $(H,L_0)$-module.
\end{enumerate}

\subsection{Original formula}\label{SubsectionOrigFormulaHPIexp}

Let $L$ be an $H$-nice Lie algebra over an algebraically closed field $F$ of characteristic $0$.
Fix some Levi decomposition $L=B\oplus R$ (direct sum of $H$-submodules).

Consider $H$-invariant ideals $I_1, I_2, \ldots, I_r$,
$J_1, J_2, \ldots, J_r$, $r \in \mathbb Z_+$, of the algebra $L$ such that $J_k \subseteq I_k$,
satisfying the conditions
\begin{enumerate}
\item $I_k/J_k$ is an irreducible $(H,L)$-module;
\item for any $H$-invariant $B$-submodules $T_k$
such that $I_k = J_k\oplus T_k$, there exist numbers
$q_i \geqslant 0$ such that $$\bigl[[T_1, \underbrace{L, \ldots, L}_{q_1}], [T_2, \underbrace{L, \ldots, L}_{q_2}], \ldots, [T_r,
 \underbrace{L, \ldots, L}_{q_r}]\bigr] \ne 0.$$
\end{enumerate}

Let $M$ be an $L$-module. Denote by $\Ann M$ its annihilator in $L$.
Let $$d(L, H) := \max \left(\dim \frac{L}{\Ann(I_1/J_1) \cap \dots \cap \Ann(I_r/J_r)}
\right)$$
where the maximum is found among all $r \in \mathbb Z_+$ and all $I_1, \ldots, I_r$, $J_1, \ldots, J_r$
satisfying Conditions 1--2. 

In~\cite[Theorem~9, see also Section~1.8]{ASGordienko5}
the following theorem is proved:

\begin{theorem}\label{TheoremMainH}
Let $L$ be a non-nilpotent $H$-nice Lie algebra over an algebraically closed field $F$ of characteristic $0$.  Then there exist constants $C_1, C_2 > 0$, $r_1, r_2 \in \mathbb R$ such that $$C_1 n^{r_1} d^n \leqslant c^{H}_n(L) \leqslant C_2 n^{r_2} d^n\text{ for all }n \in \mathbb N.$$ Here $d:=d(L, H)$.
\end{theorem}

In particular, there exists $\PIexp^H(L)=d(L, H)\in \mathbb Z_+$.

\subsection{Modification}\label{SubsectionModificationHPIexp}

Let $L$ be an $H$-nice Lie algebra.
By~\cite[Lemma~10]{ASGordienko5}, $L=B\oplus S \oplus N$
for some $H$-submodule $S\subseteq R$ such that $[B, S] = 0$.
Consider the associative subalgebra $A_0$ in $\End_F(L)$
generated by $\ad S$. Note that $A_0$ is an $H$-module algebra
since $S$ is $H$-invariant. By Condition~\ref{ConditionWedderburn}
of Subsection~\ref{SubsectionHnice},
$A_0 = \tilde A_0 \oplus J(A_0)$ (direct sum of $H$-submodules)
where $\tilde A_0$ is a maximal semisimple subalgebra of $A_0$.
(If $L$ is semisimple, $A_0=\tilde A_0=0$.)

\begin{lemma}\label{LemmaA0DirectSumOfFields}
$\tilde A_0=Fe_1 \oplus \dots \oplus Fe_q$ (direct sum of ideals)
    for some idempotents $e_i \in A_0$.
\end{lemma}
\begin{proof}
Since $R$ is solvable, by Lie's theorem, there exists a basis of $L$
such that the matrices of all operators $\ad a$, $a\in R$, are upper triangular. Denote the corresponding isomorphism $\End_F(L) \to M_s(F)$ of algebras by $\psi$ where $s := \dim L$. Since $\psi(\ad R) \subseteq UT_s(F)$,
we have $\psi(A_0) \subseteq UT_s(F)$ where $UT_s(F)$ is the associative algebra of 
upper triangular $s\times s$ matrices. However, $$UT_s(F) = Fe_{11}\oplus Fe_{22}\oplus
 \dots\oplus Fe_{ss}\oplus \tilde N$$
 where $$\tilde N := \langle e_{ij} \mid 1 \leqslant i < j \leqslant s \rangle_F$$
 is a nilpotent ideal. Since $\psi$ is an isomorphism, there is no subalgebras in $A_0$
 isomorphic to $M_2(F)$, and
  $\tilde A_0=Fe_1 \oplus \dots \oplus Fe_q$ (direct sum of ideals)
    for some idempotents $e_i \in A_0$.
\end{proof}

Since $[B,S]=0$ and $e_i$ are polynomials in $\ad a$, $a \in S$, we have $[\ad B, \tilde A_0]=0$.
The semisimplicity of $B$ implies $(\ad B) \cap \tilde A_0 = \lbrace 0 \rbrace$.
Now we treat $(\ad B)\oplus \tilde A_0$ as an $H$-module Lie algebra.

\begin{lemma}\label{LemmaLHBA0ComplReducible}
$L$ is a completely reducible $(\ad B)\oplus \tilde A_0$- and $(H, (\ad B)\oplus \tilde A_0)$-module.
\end{lemma}
\begin{proof} If $L$ is semisimple, then $L=B_1 \oplus \ldots \oplus B_s$
(direct sum of $H$-invariant ideals)
for some $H$-simple Lie algebras $B_i$ (see~\cite[Theorem~6]{ASGordienko4}),
and $L$ is a completely reducible $(H, (\ad B)\oplus \tilde A_0)$-module.

Suppose now that $L$ satisfies Conditions~\ref{ConditionRNinv}--\ref{ConditionLComplHred}
of Subsection~\ref{SubsectionHnice}. 
Note that $e_i$ are commuting diagonalizable operators
on $L$. Hence they have a common basis of eigenvectors,
and $L=\bigoplus_j W_j$ where $W_j$ are the intersections of eigenspaces of $e_i$.
Each $e_i$ commutes with the operators from $\ad B$. Thus $W_j$ are $(\ad B)$-submodules.
Recall that $B$ is semisimple. Therefore, $W_j$ is the direct sum of irreducible $(\ad B)$-submodules.
Since $e_i$ act on each $W_j$ as scalar operators, $L$ is the direct sum of
irreducible $(\ad B)\oplus \tilde A_0$-submodules. 
Now Condition~\ref{ConditionLComplHred}
of Subsection~\ref{SubsectionHnice} implies the lemma.
\end{proof}

We replace Condition~2 of Subsection~\ref{SubsectionOrigFormulaHPIexp} with Condition~2' below:
\begin{enumerate}
\item[(2')] there exist $H$-invariant $(\ad B)\oplus \tilde A_0$-submodules $T_k$, $I_k = J_k\oplus T_k$,
 and numbers
$q_i \geqslant 0$ such that $$\bigl[[T_1, \underbrace{L, \ldots, L}_{q_1}], [T_2, \underbrace{L, \ldots, L}_{q_2}], \ldots, [T_r,
 \underbrace{L, \ldots, L}_{q_r}]\bigr] \ne 0.$$
\end{enumerate}

Define $$d'(L, H) := \max \left(\dim \frac{L}{\Ann(I_1/J_1) \cap \dots \cap \Ann(I_r/J_r)}
\right)$$
where the maximum is found among all $r \in \mathbb Z_+$ and all $I_1, \ldots, I_r$, $J_1, \ldots, J_r$
satisfying Conditions 1 and 2'.

\begin{theorem}\label{Theoremddprime} Let $L$ be an $H$-nice Lie algebra over an algebraically closed field $F$ of characteristic $0$.
 Then $\PIexp^H(L)=d'(L, H)$.
\end{theorem}
\begin{proof} Clearly, $d'(L, H) \geqslant d(L, H) = \PIexp^H(L)$ since, by Lemma~\ref{LemmaLHBA0ComplReducible}, $L$ is a completely reducible $(H,(\ad B)\oplus \tilde A_0)$-module
and we can always choose $H$-invariant $(\ad B)\oplus \tilde A_0$-submodules $T_k$ such that $I_k = J_k \oplus T_k$.

If $L$ is semisimple, then~\cite[Example~7]{ASGordienko5} implies $d'(L, H) = d(L, H)$.
Hence we may assume that $L$ satisfies Conditions~\ref{ConditionRNinv}--\ref{ConditionLComplHred}
of Subsection~\ref{SubsectionHnice}. 

We prove that there exist $r \in \mathbb R$, $C > 0$ such that $c_n^H(L) \geqslant C n^r (d'(L, H))^n$
for all $n\in \mathbb N$. We
take $H$-invariant ideals $I_1, \ldots, I_r$ and $J_1, \ldots, J_r$
satisfying Conditions 1 and 2' such that
$\dim \frac{L}{\Ann(I_1/J_1) \cap \dots \cap \Ann(I_r/J_r)}
=d'(L, H)$. Then we choose $H$-invariant $(\ad B)\oplus \tilde A_0$-submodules $\tilde T_k$, $I_k = J_k\oplus \tilde T_k$,
 such that $$\bigl[[\tilde T_1, \underbrace{L, \ldots, L}_{q_1}], [\tilde T_2, \underbrace{L, \ldots, L}_{q_2}], \ldots, [\tilde T_r,
 \underbrace{L, \ldots, L}_{q_r}]\bigr] \ne 0$$
 for some numbers $q_i \geqslant 0$. Now we
 repeat the arguments of~\cite[Section~6]{ASGordienko5}
with the following changes. (We use the notation from~\cite[Section~6]{ASGordienko5}.) Instead of using Lemma~15, we choose
$c_{ij}\in \tilde A_0$ and $d_{ij}\in J(A_0)$ such that 
each $\ad a_{ij} = c_{ij}+d_{ij}$. Note that, by the second part of the proof of~\cite[Lemma~5]{ASGordienko5} for $W=S$ and $M=L$, we have $J(A_0)\subseteq J(A)$
where $A$ is the associative subalgebra of $\End_F(L)$ generated by the operators from $H$
and $\ad L$. Hence $d_{ij}\in J(A)$. 
Moreover, $\tilde T_k$ that we have chosen by Condition 2', are $H$-invariant $\tilde B$-submodules, and we use them in~\cite[Lemma~17]{ASGordienko5}.
The rest of the proof is the same as in~\cite[Section~6]{ASGordienko5}.
Finally, we have $\PIexp^H(L)\geqslant d'(L, H)$, and the theorem is proved.
\end{proof}

\section{Lie $G$-algebras and Lie algebras with derivations}\label{SectionApplLieDer}

In~\cite[Theorem~7]{GordienkoKochetov}, the authors proved the existence of the differential PI-exponent
for finite dimensional Lie algebras with an action of a finite dimensional semisimple Lie algebra by derivations. Here we prove that the differential PI-exponent coincides with the ordinary one.

\begin{theorem}\label{TheoremddprimeDiff}
Let $L$ be a finite dimensional Lie algebra
over an algebraically closed field $F$ of characteristic $0$. Suppose a
Lie algebra $\mathfrak g$ is acting on $L$ by derivations, and $L$ is an $U(\mathfrak g)$-nice algebra.
Then $\PIexp(L) = \PIexp^{U(\mathfrak g)}(L)$.
\end{theorem}
\begin{remark}
If a reductive affine algebraic group $G$ is rationally acting on $L$ by automorphisms,
then $L$ is an $FG$-nice algebra~\cite[Example~6]{ASGordienko5}. 
Hence if $G$ is connected and $\mathfrak g$ is the Lie algebra of $G$, then by~\cite[Theorems 13.1 and 13.2]{HumphreysAlgGr}, $L$ is an $U(\mathfrak g)$-nice algebra.
In particular, a finite dimensional Lie algebra $L$ with an action of a
finite dimensional semisimple Lie algebra~$\mathfrak g$ by derivations is always an $U(\mathfrak g)$-nice algebra, since there exists a simply connected semisimple affine algebraic group $G$ rationally acting on $L$ by automorphisms, such that $\mathfrak g$ is the Lie algebra of $G$
and the $\mathfrak g$-action is the differential of the $G$-action (see e.g.~\cite[Chapter XVIII, Theorem 5.1]{Hochschild} and~\cite[Theorem~3]{GordienkoKochetov}). 
\end{remark}
\begin{proof}[Proof of Theorem~\ref{TheoremddprimeDiff}]
By Theorems~\ref{TheoremMainH} and~\ref{Theoremddprime}, there exist $\PIexp(L)=d'(L, F)$
and $\PIexp^{U(\mathfrak g)}(L)=d'(L, U(\mathfrak g))$.
If we treat differential and ordinary multilinear Lie polynomials as 
multilinear functions  on $L$, we obtain $c_n(L) \leqslant c^{U(\mathfrak g)}_n(L)$ for all $n \in\mathbb N$.
Hence $\PIexp(L) \leqslant \PIexp^{U(\mathfrak g)}(L)$.

 Suppose $\mathfrak g$-invariant ideals $I_1, I_2, \ldots, I_r$,
$J_1, J_2, \ldots, J_r$, $r \in \mathbb Z_+$, of the algebra $L$ such that $J_k \subseteq I_k$,
satisfy Conditions 1 and 2' for $H=U(\mathfrak g)$. 
By Condition 2', there exist $\mathfrak g$-invariant $(\ad B)\oplus \tilde A_0$-submodules $T_k$, $I_k = J_k\oplus T_k$, and numbers $q_i \geqslant 0$ such that $$\bigl[[T_1, \underbrace{L, \ldots, L}_{q_1}], [T_2, \underbrace{L, \ldots, L}_{q_2}], \ldots, [T_r,
 \underbrace{L, \ldots, L}_{q_r}]\bigr] \ne 0.$$
 By Lemma~\ref{LemmaLHBA0ComplReducible}, $L$ is a completely reducible $(\ad B)\oplus \tilde A_0$-module.
 Hence $T_k=T_{k1}\oplus T_{k2}\oplus \ldots \oplus T_{kn_k}$
 for some irreducible $(\ad B)\oplus \tilde A_0$-submodules $T_{kj}$.
Therefore we can choose $1 \leqslant j_k \leqslant n_k$ such that
$$\bigl[[T_{1j_1}, \underbrace{L, \ldots, L}_{q_1}], [T_{2j_2}, \underbrace{L, \ldots, L}_{q_2}], \ldots, [T_{rj_r},
 \underbrace{L, \ldots, L}_{q_r}]\bigr] \ne 0.$$
 Let $\tilde I_k = T_{kj_k}\oplus J_k$.
 
 We claim that $\tilde I_k$ is an ideal in $L$ and $\Ann(\tilde I_k / J_k)=\Ann(I_k/J_k)$
 for all $1 \leqslant k \leqslant r$.
  Denote by $L_0$, $B_0$, $R_0$, $\mathfrak g_0$, respectively, the images
  of $L$, $B$, $R$, $\mathfrak g$ in $\mathfrak{gl}(I_k/J_k)$.
 Note that $B_0$ and $R_0$
 are, respectively, semisimple and solvable. Hence $L_0=B_0\oplus R_0$
 (direct sum of $\mathfrak g$-submodules) where $\mathfrak g$-action  
 on $\mathfrak{gl}(I_k/J_k)$ is induced from the $\mathfrak g$-action
 on $I_k/J_k$ and corresponds to the adjoint action of $\mathfrak g_0$
 on $\mathfrak{gl}(I_k/J_k)$. 
 In particular, $R_0$ is a solvable ideal of $(L_0+\mathfrak g_0)$
 and $B_0$ is an ideal
 of $(B_0 + \mathfrak g_0)$.
 Note that $I_k/J_k$ is an irreducible $(L_0+\mathfrak g_0)$-module. By~E.~Cartan's theorem~\cite[Proposition~1.4.11]{GotoGrosshans}, $L_0+\mathfrak g_0 = B_1
 \oplus R_1$ (direct sum of ideals) where $B_1$ is semisimple and $R_1$ is either zero
or equal to the center $Z(\mathfrak{gl}(I_k/J_k))$ consisting of scalar operators.
Considering the resulting projection $(L_0+\mathfrak g_0) \to R_1$,
we obtain $B_0 \subseteq B_1$. Since $R_0 \subseteq R_1$ consists of scalar operators,
$B_0$ is an ideal of $(L_0+\mathfrak g_0)$ and $B_1$.

    Since $\tilde I_k/J_k$ is an irreducible $(\ad B)\oplus \tilde A_0$-module and  $\tilde A_0$ is acting on $I_k/J_k$ by scalar operators,
  $\tilde I_k/J_k$ is an irreducible $B_0$- and $L$-module. In particular, $\tilde I_k$ is an ideal.
  
  If $\Ann(\tilde I_k/J_k)\ne \Ann(I_k/J_k)$, then $a\tilde I_k/J_k=0$
 for some $a \in L_0 \cong L/\Ann(I_k/J_k)$, $a\ne 0$.
 Let $\varphi \colon L_0 \to \mathfrak{gl}(\tilde I_{k}/J_k)$
be the corresponding action and $a = b + c$
 where $b\in B_0$, $c \in R_0$. 
 Then $\varphi(b)=-\varphi(c)$ is a scalar operator on $\tilde I_{k}/J_k$.
 Hence $\varphi(b)$ belongs to the center of the semisimple
 algebra $\varphi(B_0)$. Thus $\varphi(b)=\varphi(c)=0$, $b\ne 0$.
 Recall that  $B_1$ is a semisimple algebra.
 Therefore $B_1 =  B_0 \oplus B_2$ (direct sum of ideals)
 for some $B_2$. Since $R_1$ consists of scalar operators, $I_k/J_k$ is an irreducible $B_1$-module and
 we have $$I_k/J_k =  \sum_{\substack{a_i \in B_2,\\ \alpha\in\mathbb Z_+}} a_1 \ldots a_\alpha\ \tilde I_k/J_k.$$
 Now $[b, B_2]=0$ and $b\tilde I_{k}/J_k=0$  implies $bI_k/J_k=0$ and $b=0$. We get a contradiction.
  Hence $\Ann(\tilde I_k/J_k)=\Ann(I_k/J_k)$.
  
  Note that $\tilde I_1, \tilde I_2, \ldots, \tilde I_r$,
$J_1, J_2, \ldots, J_r$ satisfy Conditions 1 and 2' for $H=F$, i.e. for the case of ordinary
polynomial identities. Moreover, $$\dim \frac{L}{\Ann(I_1/J_1) \cap \dots \cap \Ann(I_r/J_r)}
= \dim \frac{L}{\Ann(\tilde I_1/J_1) \cap \dots \cap \Ann(\tilde I_r/J_r)}.$$
Hence $\PIexp^{U(\mathfrak g)}(L)=\PIexp(L)$.
\end{proof}

Analogs for associative algebras of Theorems~\ref{TheoremLieDerPIexpEqual} and~\ref{TheoremLieGConnPIexpEqual} below
 were proved in~\cite[Theorems 15 and 16]{GordienkoKochetov}. 

 \begin{theorem}\label{TheoremLieDerPIexpEqual}
Let $L$ be a finite dimensional Lie algebra
over a field $F$ of characteristic $0$. Suppose a finite dimensional
semisimple Lie algebra $\mathfrak g$ acts on $L$ by derivations.
Then $\PIexp^{U(\mathfrak g)}(L)=\PIexp(L)$.
\end{theorem}
\begin{proof} 
$H$-codimensions do not change upon an extension of the base field.
The proof is analogous to the cases of ordinary codimensions of
associative~\cite[Theorem~4.1.9]{ZaiGia} and
Lie algebras~\cite[Section~2]{ZaiLie}.
Thus without loss of generality we may assume
 $F$ to be algebraically closed.
  Now we use Theorem~\ref{TheoremddprimeDiff} and the remark after it.
\end{proof}

\begin{remark}
Theorem~\ref{TheoremLieDerPIexpEqual} implies similar asymptotic 
behavior of ordinary and differential codimensions, however the codimensions themselves may be different.
Consider the adjoint action of $\mathfrak{sl}_2(F)$ on itself. Then $c_1(\mathfrak{sl}_2(F))=1
< c_1^{U(\mathfrak{sl}_2(F))}(\mathfrak{sl}_2(F))$
since $x_1^{e_{11}-e_{22}}$ and $x_1^{e_{12}}$ are linearly independent modulo $\Id^{U(\mathfrak{sl}_2(F))}(\mathfrak{sl}_2(F))$.
\end{remark}

\begin{theorem}\label{TheoremLieGConnPIexpEqual}
Let $L$ be a finite dimensional Lie algebra
over an algebraically closed field $F$ of characteristic $0$. Suppose a connected reductive
affine algebraic group $G$ is rationally acting on $L$ by automorphisms.
Then $\PIexp^{G}(L)=\PIexp(L)$.
\end{theorem}
\begin{proof}
Note that the Lie algebra $\mathfrak g$ of the group $G$ is acting on $L$ by derivations.
By~\cite[Lemma~5]{GordienkoKochetov}, $c_n^{U(\mathfrak g)}(L)=c_n^{G}(L)$ for all $n\in\mathbb N$.
Hence Theorem~\ref{TheoremddprimeDiff} implies $\PIexp^{G}(L)=\PIexp(L)$.
\end{proof}

\begin{remark} In Theorem~\ref{TheoremLieGConnPIexpEqual} one could consider the case
when $G$ is acting by anti-automorphisms too. However, in this case $G = G_0 \cup G_1$, $G_0 \cap G_1 = \varnothing$, where the elements of $G_0$ are acting on $L$ by automorphisms and the elements
of $G_1$ are acting by anti-automorphisms. Since $G_0$ and $G_1$ are defined by polynomial
equations, they are closed subsets in $G$. Recall that $G$ is connected.
Therefore $G_1 = \varnothing$ and $G$ must act by automorphisms only.
\end{remark}

\section{Lie algebras with $R=N$}\label{SectionHRTheSame}

\subsection{Formulation of the theorem}

If the solvable radical of an $H$-module Lie algebra $L$ is nilpotent, we 
do not require from $L$ to satisfy Conditions~\ref{ConditionLevi}--\ref{ConditionLComplHred} in the definition of an $H$-nice algebra (see Subsection~\ref{SubsectionHnice}). Moreover, the formula for the Hopf PI-exponent is simpler, than in the general case (Subsections~\ref{SubsectionOrigFormulaHPIexp} and~\ref{SubsectionModificationHPIexp}).

\begin{theorem}\label{TheoremMainNRSame} Let $L$ be a finite dimensional non-nilpotent
$H$-module Lie algebra 
where $H$ is a Hopf algebra over a field $F$ of characteristic $0$.
Suppose that the solvable radical of $L$ coincides with the nilpotent
radical $N$ of $L$ and $N$ is an $H$-submodule.
Then there exist constants $d\in\mathbb N$, $C_1, C_2 > 0$, $r_1, r_2 \in \mathbb R$ such that $$C_1 n^{r_1} d^n \leqslant c^{H}_n(L) \leqslant C_2 n^{r_2} d^n\text{ for all }n \in \mathbb N.$$
Moreover, if $F$ is algebraically closed, the constant $d$ is defined as follows.
Let $$L/N = B_1 \oplus \ldots \oplus B_q \text{ (direct sum of $H$-invariant ideals)}$$
where $B_i$ are $H$-simple Lie algebras and let $\varkappa \colon L/N \to L$
be any homomorphism of algebras (not necessarily $H$-linear) such that $\pi\varkappa = \id_{L/N}$ where $\pi \colon L \to L/N$ is the natural projection. 
 Then $$d= \max\left( B_{i_1}\oplus B_{i_2} \oplus \ldots \oplus B_{i_r}
 \mathbin{\Bigl|}  r \geqslant 1,\ \bigl[[ H\varkappa(B_{i_1}), \underbrace{L, \ldots, L}_{q_1}], \right.$$ \begin{equation}\label{EqdRNSame}\left. [  H\varkappa(B_{i_2}), \underbrace{L, \ldots, L}_{q_2}], \ldots, [ H\varkappa(B_{i_r}),
 \underbrace{L, \ldots, L}_{q_r}]\bigr] \ne 0 \text{ for some } q_i \geqslant 0 \right).\end{equation}
\end{theorem}
\begin{remark}
If $L$ is nilpotent, i.e. $[x_1, \ldots, x_p]\equiv 0$ for some $p\in\mathbb N$, then
$V^{H}_n \subseteq \Id^{H}(L)$ and $c^H_n(L)=0$ for all $n \geqslant p$.
\end{remark}

Theorem~\ref{TheoremMainNRSame} will be proved at the end of Subsection~\ref{SubsectionNRSameLowerBound}.

\begin{corollary}
The analog of Amitsur's conjecture holds
 for such codimensions.
\end{corollary}
\begin{remark}
The existence of a decomposition $L/N = B_1 \oplus \ldots \oplus B_q$ (direct sum of $H$-invariant ideals)
where $B_i$ are $H$-simple Lie algebras, follows from~\cite[Theorem~6]{ASGordienko4}.
The existence of the map $\varkappa$ follows from the ordinary Levi theorem.
\end{remark}
\begin{remark}
Note that by~\cite[Lemma~9]{GordienkoKochetov},
every differential simple algebra is simple.
By~\cite[Lemma~10]{GordienkoKochetov}, a $G$-simple algebra is simple
for a rational action of a connected affine algebraic group~$G$.
Therefore, Theorem~\ref{TheoremMainNRSame} yields another proof of Theorems~\ref{TheoremLieDerPIexpEqual}
and~\ref{TheoremLieGConnPIexpEqual} for the case $R=N$ since in the conditions of the latter theorems there exists an $H$-invariant Levi decomposition and we can choose $\varkappa$ to be a homomorphism of $H$-modules.
\end{remark}
\begin{corollary}
Let $L$ be a finite dimensional non-nilpotent
Lie algebra  over a field $F$ of characteristic $0$ with an action of a group $G$ by automorphisms and anti-automorphisms. Suppose that the solvable radical of $L$ coincides with the nilpotent radical $N$ of $L$.
Then there exist constants $d\in\mathbb N$, $C_1, C_2 > 0$, $r_1, r_2 \in \mathbb R$ such that $$C_1 n^{r_1} d^n \leqslant c^{G}_n(L) \leqslant C_2 n^{r_2} d^n\text{ for all }n \in \mathbb N.$$
\end{corollary}
\begin{proof}
By~\cite[Lemma~28]{ASGordienko5}, we may assume that $G$ is acting by automorphisms only. Now we notice that
radicals are invariant under all automorphisms. Hence we may apply Theorem~\ref{TheoremMainNRSame}.
\end{proof}
\begin{corollary}
Let $L$ be a finite dimensional non-nilpotent
Lie algebra  over a field $F$ of characteristic $0$ with an action of a Lie algebra $\mathfrak g$ by derivations. Suppose that the solvable radical of $L$ coincides with the nilpotent radical $N$ of $L$. Then there exist constants $d\in\mathbb N$, $C_1, C_2 > 0$, $r_1, r_2 \in \mathbb R$ such that $$C_1 n^{r_1} d^n \leqslant c^{U(\mathfrak g)}_n(L) \leqslant C_2 n^{r_2} d^n\text{ for all }n \in \mathbb N.$$
\end{corollary}
\begin{proof}
By~\cite[Chapter~III, Section~6, Theorem~7]{JacobsonLie}, the
radical is invariant under all derivations. Hence we may apply Theorem~\ref{TheoremMainNRSame}.
\end{proof}

The algebra in the example below has no $G$-invariant Levi decomposition (see~\cite[Example~12]{ASGordienko4}), however it satisfies the analog of Amitsur's conjecture.

\begin{example}[Yuri Bahturin]\label{ExampleGnoninvLevi}
Let $F$ be a field of characteristic $0$ and let $$L = \left\lbrace\left(\begin{array}{cc} C & D \\
0 & 0
  \end{array}\right) \mathrel{\biggl|} C \in \mathfrak{sl}_m(F), D\in M_m(F)\right\rbrace
  \subseteq \mathfrak{sl}_{2m}(F),$$ $m \geqslant 2$.
      Consider $\varphi \in \Aut(L)$ where
  $$\varphi\left(\begin{array}{cc} C & D \\
0 & 0
  \end{array}\right)=\left(\begin{array}{cc} C & C+D \\
0 & 0
  \end{array}\right).$$
  Then $L$ is a Lie algebra with an action of the group $G=\langle \varphi \rangle
  \cong \mathbb Z$ by automorphisms
  and there exist constants $C_1, C_2 > 0$, $r_1, r_2 \in \mathbb R$ such that $$C_1 n^{r_1} (m^2-1)^n \leqslant c^G_n(L) \leqslant C_2 n^{r_2} (m^2-1)^n\text{ for all }n \in \mathbb N.$$
\end{example}
\begin{proof}
$G$-codimensions do not change upon an extension of the base field.
The proof is analogous to the cases of ordinary codimensions of
associative~\cite[Theorem~4.1.9]{ZaiGia} and
Lie algebras~\cite[Section~2]{ZaiLie}. Moreover, upon an extension of $F$, $L$ 
remains the algebra of the same type.
Thus without loss of generality we may assume
 $F$ to be algebraically closed.
 
 Note that
 $$N=\left\lbrace\left(\begin{array}{cc} 0 & D \\
0 & 0
  \end{array}\right) \mathrel{\biggl|} D\in M_m(F)\right\rbrace
  $$
  is the solvable (and nilpotent) radical of $L$ and $L/N \cong \mathfrak{sl}_m(F)$
  is a simple Lie algebra. Hence $\PIexp^G(L)=\dim\mathfrak{sl}_m(F)= m^2-1$
  by Theorem~\ref{TheoremMainNRSame}.
\end{proof}

The algebra in the example below has no $L$-invariant Levi decomposition (see~\cite[Example~13]{ASGordienko4}), however it satisfies the analog of Amitsur's conjecture.

\begin{example}
Let $L$ be the Lie algebra from Example~\ref{ExampleGnoninvLevi}.
      Consider the adjoint action of $L$ on itself by derivations.
        Then there exist constants $C_1, C_2 > 0$, $r_1, r_2 \in \mathbb R$ such that $$C_1 n^{r_1} (m^{2}-1)^n \leqslant c^{U(L)}_n(L) \leqslant C_2 n^{r_2} (m^2-1)^{n}\text{ for all }n \in \mathbb N.$$
\end{example}
\begin{proof} Again, without loss of generality we may assume
 $F$ to be algebraically closed. Since $L/N \cong \mathfrak{sl}_m(F)$
  is a simple Lie algebra, $\PIexp^{U(L)}(L)=\dim\mathfrak{sl}_m(F)= m^2-1$
  by Theorem~\ref{TheoremMainNRSame}.
\end{proof}

\subsection{$S_n$-cocharacters and upper bound}

One of the main tools in the investigation of polynomial
identities is provided by the representation theory of symmetric groups.

Let $L$ be an $H$-module Lie algebra
over a field $F$ of characteristic $0$.
 The symmetric group $S_n$  acts
 on the spaces $\frac {V^H_n}{V^H_{n}
  \cap \Id^H(L)}$
  by permuting the variables.
  Irreducible $FS_n$-modules are described by partitions
  $\lambda=(\lambda_1, \ldots, \lambda_s)\vdash n$ and their
  Young diagrams $D_\lambda$.
   The character $\chi^H_n(L)$ of the
  $FS_n$-module $\frac {V^H_n}{V^H_n
   \cap \Id^H(L)}$ is
   called the $n$th
  \textit{cocharacter} of polynomial $H$-identities of $L$.
   We can rewrite $\chi^H_n(L)$ as
  a sum $$\chi^H_n(L)=\sum_{\lambda \vdash n}
   m(L, H, \lambda)\chi(\lambda)$$ of
  irreducible characters $\chi(\lambda)$.
Let  $e_{T_{\lambda}}=a_{T_{\lambda}} b_{T_{\lambda}}$
and
$e^{*}_{T_{\lambda}}=b_{T_{\lambda}} a_{T_{\lambda}}$
where
$a_{T_{\lambda}} = \sum_{\pi \in R_{T_\lambda}} \pi$
and
$b_{T_{\lambda}} = \sum_{\sigma \in C_{T_\lambda}}
 (\sign \sigma) \sigma$,
be Young symmetrizers corresponding to a Young tableau~$T_\lambda$.
Then $M(\lambda) = FS e_{T_\lambda} \cong FS e^{*}_{T_\lambda}$
is an irreducible $FS_n$-module corresponding to
 a partition~$\lambda \vdash n$.
  We refer the reader to~\cite{Bahturin, DrenKurs, ZaiGia}
   for an account
  of $S_n$-representations and their applications to polynomial
  identities.

In the next two lemmas we consider
 a finite dimensional 
$H$-module Lie algebra $L$ with an $H$-invariant nilpotent ideal $N$
where  $H$ is a Hopf algebra over a field $F$ of characteristic $0$
and $N^p=0$
for some $p\in\mathbb N$.
Fix a decomposition $L/N = B_1 \oplus \ldots \oplus B_q $
where $B_i$ are some subspaces.
 Let $\varkappa \colon L/N \to L$
be an $F$-linear map such that $\pi\varkappa = \id_{L/N}$ where $\pi \colon L \to L/N$
is the natural projection. Define the number $d$ by~(\ref{EqdRNSame}).

\begin{lemma}\label{LemmaNRSameUpperCochar}
 Let $n\in\mathbb N$ and $\lambda = (\lambda_1, \ldots, \lambda_s) \vdash n$. Then if $\sum_{k=d+1}^s \lambda_k \geqslant p$, we have $m(L, H, \lambda)=0$. 
\end{lemma}
\begin{proof}
It is sufficient to prove that $e^{*}_{T_\lambda}f \in \Id^H(L)$ for all $f \in V_n$ and for all Young tableaux $T_\lambda$ corresponding to $\lambda$.

Fix a basis in $L$ that is a union of bases of~$\varkappa(B_1),\ldots, \varkappa(B_q)$ and~$N$.
Since $e^{*}_{T_\lambda}f$ is multilinear, it sufficient to prove that $e^{*}_{T_\lambda}f$
vanishes under all evaluations on basis elements.
Fix some substitution of basis elements and choose $1 \leqslant i_1,\ldots,i_r \leqslant q$
such that all the elements substituted belong to $\varkappa(B_{i_1})\oplus \ldots \oplus \varkappa(B_{i_r}) \oplus N$, and for each $j$ we have an element being substituted from $\varkappa(B_{i_j})$.
Then we may assume that $\dim(B_{i_1}\oplus \ldots \oplus B_{i_r}) \leqslant d$,
since otherwise $e^{*}_{T_\lambda}f$ is zero by the definition of $d$.
 Note that
$e^{*}_{T_\lambda} = b_{T_\lambda} a_{T_\lambda}$
and $b_{T_\lambda}$ alternates the variables of each column
of $T_\lambda$. Hence if $e^{*}_{T_\lambda} f$ does not vanish, this implies that different basis elements
are substituted for the variables of each column. 
Therefore, at least $\sum_{k=d+1}^s \lambda_k \geqslant p$ elements must be taken from $N$.
Since $N^p = 0$, we have $e^{*}_{T_\lambda} f \in \Id^H(L)$.
\end{proof}

\begin{lemma}\label{LemmaNRSameUpper} 
If $d > 0$, then there exist constants $C_2 > 0$, $r_2 \in \mathbb R$
such that $c^H_n(L) \leqslant C_2 n^{r_2} d^n$
for all $n \in \mathbb N$. In the case $d=0$, the algebra $L$ is nilpotent.
\end{lemma}
\begin{proof}
Lemma~\ref{LemmaNRSameUpperCochar} and~\cite[Lemmas~6.2.4, 6.2.5]{ZaiGia}
imply
$$
\sum_{m(L,H, \lambda)\ne 0} \dim M(\lambda) \leqslant C_3 n^{r_3} d^n
$$
for some constants $C_3, r_3 > 0$.
Together with \cite[Theorem~12]{ASGordienko5} this inequality yields the upper bound.
\end{proof}

\subsection{Lower bound}\label{SubsectionNRSameLowerBound}

Lemma~\ref{LemmaHRSameLowerPolynomial} below is a version of~\cite[Lemma~20]{ASGordienko5}
adapted for our case.

\begin{lemma}\label{LemmaHRSameLowerPolynomial}
Suppose that $F$ is an algebraically closed field of charactreistic $0$ and
let $L$, $N$, $\varkappa$, $B_i$, and $d$ be the same as in Theorem~\ref{TheoremMainNRSame}.
If $d > 0$, then there exists a number $n_0 \in \mathbb N$ such that for every $n\geqslant n_0$
there exist disjoint subsets $X_1$, \ldots, $X_{2k} \subseteq \lbrace x_1, \ldots, x_n
\rbrace$, $k := \left[\frac{n-n_0}{2d}\right]$,
$|X_1| = \ldots = |X_{2k}|=d$ and a polynomial $f \in V^H_n \backslash
\Id^H(L)$ alternating in the variables of each set $X_j$.
\end{lemma}
\begin{proof} Without lost of generality,
we may assume that $d = \dim(B_1 \oplus B_2 \oplus \ldots \oplus B_r)$
where 
$\bigl[[  H\varkappa(B_1), a_{11}, \ldots, a_{1q_1}], [  H\varkappa(B_2), a_{21}, \ldots, a_{2q_2}], \ldots, [ H\varkappa(B_r),
 a_{r1}, \ldots, a_{rq_r}]\bigr] \ne 0$ for some $q_i\geqslant 0$
 and $a_{kj}\in L$. 
 Since $N$ is nilpotent, we can increase $q_i$ adding to $\lbrace a_{ij} \rbrace$
 sufficiently many elements of $N$ such that
$$\bigl[[ {\gamma_1}\varkappa(b_1), a_{11}, \ldots, a_{1q_1}], [ {\gamma_2}\varkappa(b_2), a_{21}, \ldots, a_{2q_2}], \ldots, [{\gamma_r}\varkappa(b_r),
 a_{r1}, \ldots, a_{rq_r}]\bigr] \ne 0$$ for some $q_i\geqslant 0$, $b_i \in B_i$, $\gamma_i \in H$,
 however
 \begin{equation}\label{Eqbazero}\bigl[[ \tilde b_1, a_{11}, \ldots, a_{1q_1}], [ \tilde b_2, a_{21}, \ldots, a_{2q_2}], \ldots, [\tilde b_r,
 a_{r1}, \ldots, a_{rq_r}]\bigr] = 0 \end{equation} for all $t_i \geqslant 0$, $\tilde b_i\in [H\varkappa(B_i), \underbrace{L, \ldots, L}_{t_i}]$
 such that $\tilde b_j\in [H\varkappa(B_j), L, \ldots, L, N, L, \ldots, L]$ for at least one $j$.
 
 Recall that $\varkappa$ is a homomorphism of algebras.
Moreover $\pi(h\varkappa(a)-\varkappa(ha))=0$ 
implies $h\varkappa(a)-\varkappa(ha) \in N$ for all $a\in L$ and $h\in H$.
Hence, by~(\ref{Eqbazero}), if we replace $\varkappa(b_i)$ in $$\bigl[[ \gamma_1\varkappa(b_1), a_{11}, \ldots, a_{1q_1}], [  \gamma_2\varkappa(b_2), a_{21}, \ldots, a_{2q_2}], \ldots, [ \gamma_r\varkappa(b_r),
 a_{r1}, \ldots, a_{rq_r}]\bigr]$$ with the commutator of $\varkappa(b_i)$ and an expression involving $\varkappa$, the map $\varkappa$
 will behave like a homomorphism of $H$-modules.
 We will exploit this property further.

In virtue of~\cite[Theorem~11]{ASGordienko5},
there exist constants $m_i \in \mathbb Z_+$
such that for any $k$ there exist
 multilinear associative $H$-polynomials $f_i$ of degree $(2kd_i + m_i)$,
$d_i := \dim B_i$,
alternating in the variables from disjoint sets
$X^{(i)}_{\ell}$, $1 \leqslant \ell \leqslant 2k$, $|X^{(i)}_{\ell}|=d_i$,
such that each $f_i$ does not vanish under some evaluation in $\ad B_i$.

Since $B_i$ is an irreducible $(H, \ad B_i)$-module, by the Density Theorem,
 $\End_F(B_i)$ is generated by the operators from~$H$ and~$\ad B_i$.
  Note that $\End_F(B_i) \cong M_{d_i}(F)$.
Thus every matrix unit $e^{(i)}_{j\ell} \in M_{d_i}(F)$ can be
represented as a polynomial in operators from $H$
and $\ad B_i$. Choose such polynomials
for all $i$ and all matrix units. Denote by $m_0$ the maximal degree of those
polynomials.

Let $n_0 := r(2m_0+1)+ \sum_{i=1}^r (m_i+q_i)$.
Now we choose $f_i$ for $k = \left[\frac{n-n_0}{2d}\right]$.
In addition, we choose $\tilde f_1$ for $\tilde k = \left[\frac{n-2kd-m_1}{2d_1}\right]+1$
and $B_{i_1}$ using~\cite[Theorem~11]{ASGordienko5} once again. The polynomials $f_i$ will deliver us the required alternations. However, the total degree of the product may be less than $n$. We will use $\tilde f_1$ to increase the number of variables and obtain a polynomial of degree $n$.

By~\cite[Theorem~11]{ASGordienko5},
there exist $\bar x_{i1}, \ldots, \bar x_{i, 2k d_i+m_i} \in B_i$
such that $$f_i(\ad \bar x_{i1}, \ldots, \ad \bar x_{i, 2k d_i+m_i})\ne 0,$$
and $\bar x_1, \ldots, \bar x_{2\tilde k d_1+m_1} \in B_1$
such that $\tilde f_1(\ad \bar x_1, \ldots, \ad \bar x_{2\tilde k d_1+m_1}) \ne 0$.
Hence $$e^{(i)}_{\ell_i \ell_i} f_i(\ad \bar x_{i1}, \ldots, \ad \bar x_{i, 2k d_i+m_i})
e^{(i)}_{s_i s_i} \ne 0$$
and $$e^{(1)}_{\tilde\ell \tilde\ell}\tilde f_1(\ad \bar x_1, \ldots, \ad \bar x_{2\tilde k d_1+m_1})
e^{(1)}_{\tilde s \tilde s} \ne 0$$
 for some matrix units $e^{(i)}_{\ell_i \ell_i},
e^{(i)}_{s_i s_i} \in \End_F(B_i)$, $1 \leqslant \ell_i, s_i \leqslant d_i$,
$e^{(1)}_{\tilde\ell \tilde\ell}, e^{(1)}_{\tilde s \tilde s} \in \End_F(B_1)$, $1 \leqslant \tilde \ell,
\tilde s \leqslant d_1$.
Thus $$\sum_{\ell=1}^{d_i}
e^{(i)}_{\ell \ell_i} f_i(\bar x_{i1}, \ldots, \bar x_{i, 2k d_i+m_i})
 e^{(i)}_{s_i \ell}$$ is a nonzero scalar operator in $\End_F(B_i)$.

Hence
$$ [[\gamma_1\varkappa\left(\sum_{\ell=1}^{d_1}
e^{(1)}_{\ell \ell_1} f_1(\ad \bar x_{11}, \ldots, \ad \bar x_{1,2k d_1+m_1})
e^{(1)}_{s_1 \tilde \ell} \tilde f_1(\ad \bar x_1, \ldots, \ad \bar x_{2\tilde k d_1+m_1})
 e^{(1)}_{\tilde s \ell}b_1\right), a_{11}, \ldots, a_{1q_1}],$$
 $$ [\gamma_2\varkappa\left(\sum_{\ell=1}^{d_2}
e^{(2)}_{\ell \ell_2} f_2(\ad \bar x_{21}, \ldots, \ad \bar x_{2,2k d_2+m_2})
 e^{(2)}_{s_2 \ell}b_2\right), a_{21}, \ldots, a_{2q_2}],
 \ldots, $$
 $$
 [\gamma_r\varkappa\left(\sum_{\ell=1}^{d_r}
e^{(r)}_{\ell \ell_r} f_r(\ad \bar x_{r1}, \ldots, \ad \bar x_{r, 2k d_r+m_r})
 e^{(r)}_{s_r \ell}b_r\right), a_{r1}, \ldots, a_{rq_r}]]\ne 0.$$

Now we rewrite
$e^{(i)}_{\ell j}$ as polynomials in elements of $\ad B_i$
and $H$.
Using linearity of the expression in $e^{(i)}_{\ell j}$,
we can replace $e^{(i)}_{\ell j}$ with the products
of elements from $\ad B_i$
and $H$, and the expression will not vanish
for some choice of the products. By the definition 
of an $H$-module algebra,
$h(\ad a )b=\ad (h_{(1)}a)(h_{(2)}b)$
for all $h\in H$ and $a, b \in L$. Hence
we can move all elements from $H$ to the right.
As we have mentioned, $\varkappa$ is a homomorphism of algebras and, by~(\ref{Eqbazero}), behaves like
a homomorphism of $H$-modules. Hence
we get
$$  a_0 := \biggl[\Bigl[\gamma_1\Bigl[\bar y_{11}, [\bar y_{12}, \ldots
 [\bar y_{1 \alpha_1}, 
$$ $$
  (f_1(\ad \varkappa(\bar x_{11}), \ldots, \ad \varkappa(\bar x_{1, 2k d_1+m_1})))^{h_1}
 [\bar w_{11}, [\bar w_{12}, \ldots, [\bar w_{1 \theta_1},$$ $$
 (\tilde f_1(\ad \varkappa(\bar x_1), \ldots, \ad \varkappa(\bar x_{2\tilde k d_1+m_1})))^{\tilde h}
 [\bar w_{1}, [\bar w_{2}, \ldots, [\bar w_{\tilde \theta},
  \varkappa({h'_1}b_1)]\ldots \Bigr],
  a_{11}, \ldots, a_{1q_1}\Bigr],$$
  $$\Bigl[\gamma_2\Bigl[\bar y_{21}, [\bar y_{22}, \ldots
 [\bar y_{2 \alpha_2}, $$ $$
  (f_2(\ad \varkappa(\bar x_{21}), \ldots, \ad \varkappa(\bar x_{2, 2k d_2+m_2})))^{h_2}
 [\bar w_{21}, [\bar w_{22}, \ldots, [\bar w_{2 \theta_2},
  \varkappa({h'_2}b_2)]\ldots \Bigr],
  a_{21}, \ldots, a_{2q_2}\Bigr],
 \ldots, $$
 $$\Bigl[\gamma_r\Bigl[\bar y_{r1}, [\bar y_{r2}, \ldots,
 [\bar y_{r \alpha_r},  $$ $$
 (f_r(\ad \varkappa(\bar x_{r1}), \ldots, \ad \varkappa(\bar x_{r, 2k d_r+m_r})))^{h_r}
 [\bar w_{r1}, [\bar w_{r2}, \ldots, [\bar w_{r \theta_r}, \varkappa({h'_r}b_r)]\ldots \Bigr],
  a_{r1}, \ldots, a_{rq_r}\Bigr]\biggr]\ne 0$$
 for some  $0 \leqslant \alpha_i, \theta_i, \tilde \theta \leqslant m_0$,
  \quad $h_i, h'_i, \tilde h \in H$,\quad $\bar y_{ij}, \bar w_{ij} \in \varkappa(B_i)$,
  \quad $\bar w_j \in \varkappa(B_1)$.
 
 We assume that each $f_i$ is a polynomial in $x_{i1}, \ldots,
x_{i,2k d_i+m_i}$ and $\tilde f_1$ is a polynomial in $x_1, \ldots, x_{2\tilde k d_1 + m_1}$.
Denote $X_\ell := \bigcup_{i=1}^{r} X^{(i)}_{\ell}$
where $f_i$ is alternating in the variables of each $X^{(i)}_{\ell}$.
Let $\Alt_\ell$ be the operator of alternation
in the variables from $X_\ell$. 

Consider
$$\hat f :=
 \Alt_1 \Alt_2 \ldots \Alt_{2k} \biggl[\Bigl[\gamma_1\Bigl[y_{11}, [y_{12}, \ldots
 [y_{1 \alpha_1}, 
$$ $$
  (f_1(\ad x_{11}, \ldots, \ad x_{1, 2k d_1+m_1}))^{h_1}
 [w_{11}, [w_{12}, \ldots, [w_{1 \theta_1},$$ $$
 (\tilde f_1(\ad x_1, \ldots, \ad x_{2\tilde k d_1+m_1}))^{\tilde h}
 [w_{1}, [w_{2}, \ldots, [w_{\tilde \theta},
  z_1]\ldots \Bigr],
  u_{11}, \ldots, u_{1q_1}\Bigr],$$
  $$\Bigl[\gamma_2\Bigl[y_{21}, [y_{22}, \ldots
 [y_{2 \alpha_2}, $$ $$
  (f_2(\ad x_{21}, \ldots, \ad x_{2, 2k d_2+m_2}))^{h_2}
 [w_{21}, [w_{22}, \ldots, [w_{2 \theta_2},
  z_2]\ldots \Bigr],
  u_{21}, \ldots, u_{2q_2}\Bigr],
 \ldots, $$
 $$\Bigl[\gamma_r\Bigl[y_{r1}, [y_{r2}, \ldots,
 [y_{r \alpha_r},  $$ $$
 (f_r(\ad x_{r1}, \ldots, \ad x_{r, 2k d_r+m_r}))^{h_r}
 [w_{r1}, [w_{r2}, \ldots, [w_{r \theta_r}, z_r]\ldots \Bigr],
  u_{r1}, \ldots, u_{rq_r}\Bigr]\biggr].$$

Then the value of $\hat f$
under the substitution
$z_i=\varkappa({h'_i}b_i)$, $u_{i\ell}=a_{i\ell}$,
 $x_{i\ell}=\varkappa(\bar x_{i\ell})$, $x_i = \varkappa(\bar x_i)$, $y_{i\ell}=\bar y_{i\ell}$, $w_{i\ell}=\bar w_{i\ell}$, $w_i = \bar w_i$
 equals $(d_1!)^{2k} \ldots (d_r!)^{2k} a_0 \ne 0$
since $f_i$ are alternating in the variables of each $X^{(i)}_{\ell}$, $[B_i, B_\ell] = 0$ for $i \ne \ell$,
and $\varkappa$ is a homomorphism of algebras.
 
 Hence $$f_0 := 
  \Alt_1 \Alt_2 \ldots \Alt_{2k} \biggl[\Bigl[\gamma_1\Bigl[y_{11}, [y_{12}, \ldots
 [y_{1 \alpha_1}, 
$$ $$
  (f_1(\ad x_{11}, \ldots, \ad x_{1, 2k d_1+m_1}))^{h_1}
 [w_{11}, [w_{12}, \ldots, [w_{1 \theta_1},
  z_1]\ldots \Bigr],
  u_{11}, \ldots, u_{1q_1}\Bigr],$$
  $$\Bigl[\gamma_2\Bigl[y_{21}, [y_{22}, \ldots
 [y_{2 \alpha_2}, $$ $$
  (f_2(\ad x_{21}, \ldots, \ad x_{2, 2k d_2+m_2}))^{h_2}
 [w_{21}, [w_{22}, \ldots, [w_{2 \theta_2},
  z_2]\ldots \Bigr],
  u_{21}, \ldots, u_{2q_2}\Bigr],
 \ldots, $$
 $$\Bigl[\gamma_r\Bigl[y_{r1}, [y_{r2}, \ldots,
 [y_{r \alpha_r},  $$ $$
 (f_r(\ad x_{r1}, \ldots, \ad x_{r, 2k d_r+m_r}))^{h_r}
 [w_{r1}, [w_{r2}, \ldots, [w_{r \theta_r}, z_r]\ldots \Bigr],
  u_{r1}, \ldots, u_{rq_r}\Bigr]\biggr]$$
   does not vanish under the substitution
 $$z_1 = (\tilde f_1(\ad \varkappa(\bar x_1), \ldots, \ad \varkappa(\bar x_{2\tilde k d_1+m_1})))^{\tilde h}
 [\bar w_{1}, [\bar w_{2}, \ldots, [\bar w_{\tilde \theta},
 \varkappa(h'_1 b_1)]\ldots],$$
  $z_i=\varkappa(h'_i b_i)$ for $2 \leqslant i \leqslant r$; $u_{i\ell}=a_{i\ell}$,
 $x_{i\ell}=\varkappa(\bar x_{i\ell})$, $y_{i\ell}=\bar y_{i\ell}$, $w_{i\ell}=\bar w_{i\ell}$.
 
Note that $f_0 \in V_{\tilde n}^H$,
  $\tilde n: = 2kd +r+ \sum_{i=1}^r (m_i + q_i + \alpha_i+\theta_i)
  \leqslant n$. If $n=\tilde n$, then we take $f:=f_0$.
  Suppose $n > \tilde n$.
Note that $(\tilde f_1(\ad \varkappa(\bar x_1), \ldots, \ad \varkappa(\bar x_{2\tilde k d_1+m_1})))^{\tilde h}
 [\bar w_{1}, [\bar w_{2}, \ldots, [\bar w_{\tilde \theta},
  \varkappa(h'_1 b_1)]\ldots]$ is a linear combination of long commutators.
  Each of these commutators contains at least $2\tilde k d_1+m_1+1 > n-\tilde n+1$
  elements of $L$.
       Hence $ f_0$ does not vanish under a substitution
 $z_1 = [\bar v_1, [\bar v_2, [\ldots, [\bar v_\theta,  \varkappa(h'_1 b_1)]\ldots]$
 for some $\theta \geqslant n-\tilde n$, $\bar v_i \in L$;
  $z_i=\varkappa(h'_i b_i)$ for $2 \leqslant i \leqslant r$; $u_{i\ell}=a_{i\ell}$,
 $x_{i\ell}=\varkappa(\bar x_{i\ell})$, $y_{i\ell}=\bar y_{i\ell}$,
  $w_{i\ell}=\bar w_{i\ell}$.
Therefore, $$f :=  \Alt_1 \Alt_2 \ldots \Alt_{2k} \biggl[\Bigl[\gamma_1\Bigl[y_{11}, [y_{12}, \ldots
 [y_{1 \alpha_1}, 
$$ $$
  (f_1(\ad x_{11}, \ldots, \ad x_{1, 2k d_1+m_1}))^{h_1}
 [w_{11}, [w_{12}, \ldots, [w_{1 \theta_1},$$
$$ 
 \bigl[v_1, [v_2, [\ldots, [v_{n-\tilde n}, z_1]\ldots\bigr]\ldots \Bigr],
  u_{11}, \ldots, u_{1q_1}\Bigr],$$
  $$\Bigl[\gamma_2\Bigl[y_{21}, [y_{22}, \ldots
 [y_{2 \alpha_2}, $$ $$
  (f_2(\ad x_{21}, \ldots, \ad x_{2, 2k d_2+m_2}))^{h_2}
 [w_{21}, [w_{22}, \ldots, [w_{2 \theta_2},
  z_2]\ldots \Bigr],
  u_{21}, \ldots, u_{2q_2}\Bigr],
 \ldots, $$
 $$\Bigl[\gamma_r\Bigl[y_{r1}, [y_{r2}, \ldots,
 [y_{r \alpha_r},  $$ $$
 (f_r(\ad x_{r1}, \ldots, \ad x_{r, 2k d_r+m_r}))^{h_r}
 [w_{r1}, [w_{r2}, \ldots, [w_{r \theta_r}, z_r]\ldots \Bigr],
  u_{r1}, \ldots, u_{rq_r}\Bigr]\biggr]$$
  does not vanish under the substitution
  $v_\ell = \bar v_\ell$, $1 \leqslant \ell \leqslant n-\tilde n$,
  $$z_1 = [\bar v_{n-\tilde n +1}, [\bar v_{n-\tilde n +2}, [\ldots, [\bar v_\theta,  \varkappa(h'_1 b_1)]\ldots];$$
  $z_i=\varkappa(h'_i b_i)$ for $2 \leqslant i \leqslant r$; $u_{i\ell}=a_{i\ell}$,
 $x_{i\ell}=\varkappa(\bar x_{i\ell})$, $y_{i\ell}=\bar y_{i\ell}$, $w_{i\ell}=\bar w_{i\ell}$.
 Note that $f \in V_n^H$ and satisfies all the conditions of the lemma.
\end{proof}

Lemma~\ref{LemmaHRSameCochar} is an analog of~\cite[Lemma~21]{ASGordienko5}.

\begin{lemma}\label{LemmaHRSameCochar} Let
 $k, n_0$ be the numbers from
Lemma~\ref{LemmaHRSameLowerPolynomial}. Then for every $n \geqslant n_0$ there exists
a partition $\lambda = (\lambda_1, \ldots, \lambda_s) \vdash n$,
$\lambda_i \geqslant 2k-p$ for every $1 \leqslant i \leqslant d$,
with $m(L, H, \lambda) \ne 0$.
Here $p \in \mathbb N$ is such a number that $N^p=0$.
\end{lemma}
\begin{proof}
Consider the polynomial $f$ from Lemma~\ref{LemmaHRSameLowerPolynomial}.
It is sufficient to prove that $e^*_{T_\lambda} f \notin \Id^H(L)$
for some tableau $T_\lambda$ of the desired shape $\lambda$.
It is known that $FS_n = \bigoplus_{\lambda,T_\lambda} FS_n e^{*}_{T_\lambda}$ where the summation
runs over the set of all standard tableax $T_\lambda$,
$\lambda \vdash n$. Thus $FS_n f = \sum_{\lambda,T_\lambda} FS_n e^{*}_{T_\lambda}f
\not\subseteq \Id^H(L)$ and $e^{*}_{T_\lambda} f \notin \Id^H(L)$ for some $\lambda \vdash n$.
We claim that $\lambda$ is of the desired shape.
It is sufficient to prove that
$\lambda_d \geqslant 2k-p$, since
$\lambda_i \geqslant \lambda_d$ for every $1 \leqslant i \leqslant d$.
Each row of $T_\lambda$ includes numbers
of no more than one variable from each $X_i$,
since $e^{*}_{T_\lambda} = b_{T_\lambda} a_{T_\lambda}$
and $a_{T_\lambda}$ is symmetrizing the variables of each row.
Thus $\sum_{i=1}^{d-1} \lambda_i \leqslant 2k(d-1) + (n-2kd) = n-2k$.
In virtue of Lemma~\ref{LemmaNRSameUpperCochar},
$\sum_{i=1}^d \lambda_i \geqslant n-p$. Therefore
$\lambda_d \geqslant 2k-p$.
\end{proof}

\begin{proof}[Proof of Theorem~\ref{TheoremMainNRSame}]
Let $K \supset F$ be an extension of the field $F$.
Then $$(L \otimes_F K)/(N \otimes_F K) \cong (L/N) \otimes_F K$$
is again a semisimple Lie algebra and $N \otimes_F K$ is still nilpotent.
As we have already mentioned, $H$-codimensions do not change upon an extension of $F$.
Hence we may assume $F$ to be algebraically closed.

 The Young diagram~$D_\lambda$ from Lemma~\ref{LemmaHRSameCochar} contains
the rectangular subdiagram~$D_\mu$, $\mu=(\underbrace{2k-p, \ldots, 2k-p}_d)$.
The branching rule for $S_n$ implies that if we consider the restriction of
$S_n$-action on $M(\lambda)$ to $S_{n-1}$, then
$M(\lambda)$ becomes the direct sum of all non-isomorphic
$FS_{n-1}$-modules $M(\nu)$, $\nu \vdash (n-1)$, where each $D_\nu$ is obtained
from $D_\lambda$ by deleting one box. In particular,
$\dim M(\nu) \leqslant \dim M(\lambda)$.
Applying the rule $(n-d(2k-p))$ times, we obtain $\dim M(\mu) \leqslant \dim M(\lambda)$.
By the hook formula, $$\dim M(\mu) = \frac{(d(2k-p))!}{\prod_{i,j} h_{ij}}$$
where $h_{ij}$ is the length of the hook with edge in $(i, j)$.
By Stirling formula,
$$c_n^H(L)\geqslant \dim M(\lambda) \geqslant \dim M(\mu) \geqslant \frac{(d(2k-p))!}{((2k-p+d)!)^d}
\sim $$ $$\frac{
\sqrt{2\pi d(2k-p)} \left(\frac{d(2k-p)}{e}\right)^{d(2k-p)}
}
{
\left(\sqrt{2\pi (2k-p+d)}
\left(\frac{2k-p+d}{e}\right)^{2k-p+d}\right)^d
} \sim C_4 k^{r_4} d^{2kd}$$
for some constants $C_4 > 0$, $r_4 \in \mathbb Q$,
as $k \to \infty$.
Since $k = \left[\frac{n-n_0}{2d}\right]$,
this gives the lower bound.
The upper bound has been proved in Lemma~\ref{LemmaNRSameUpper}.
\end{proof}

\section*{Acknowledgements}

This work started while I was an AARMS postdoctoral fellow at Memorial University of Newfoundland, whose faculty and staff I would like to thank for hospitality. I am grateful to Yuri Bahturin, who suggested that I study polynomial $H$-identities, and to Mikhail Zaicev, who suggested that I consider algebras without
an $H$-invariant Levi decomposition. In addition, I appreciate Mikhail Kochetov for helpful
discussions. Finally, I am grateful to the referee who found several misprints and gave some useful advices on how to improve the exposition.

\end{document}